\documentclass[12pt,reqno]{amsart}

\usepackage{fullpage}
\usepackage[utf8]{inputenc}
\usepackage{amssymb}
\usepackage{amsmath}
\usepackage{amsrefs}
\usepackage{amsthm}
\usepackage{mathrsfs}
\usepackage{changepage}
\usepackage{thmtools}
\usepackage{extpfeil}
\usepackage[colorlinks=true,linkcolor=blue,urlcolor=blue, citecolor=blue]{hyperref}
\usepackage{amsfonts}
\usepackage{times}
\usepackage{graphicx}
\usepackage{thmtools}
\usepackage{thm-restate}
\usepackage{cleveref}
\usepackage{tikz}
\newcommand{\RR}{\mathbb{R}}
\newcommand{\FF}{\mathbb{F}}
\newcommand{\ZZ}{\mathbb{Z}}

\newtheorem{theorem}{Theorem}[subsection]

\newtheorem{prop}[theorem]{Proposition}
\newtheorem{cor}[theorem]{Corollary}
\newtheorem{definition}[theorem]{Definition}
\newtheorem{remark}[theorem]{Remark}

\begin{document}

\title{Classifying toric 3-fold codes of dimensions 4 and 5}

\author[T.~Braun]{Tori Braun}
\address{Tori Braun \\ Ripon College \\ Ripon, Wisconsin \href{mailto:toribraun15@gmail.com}
{{\ttfamily\upshape toribraun15@gmail.com}}}

\author[J.~Carzon]{James Carzon}
\address{James Carzon\\ University
  of Michigan--Dearborn \\ Dearborn, Michigan
  \href{mailto:jcarzon@umich.edu}
{{\ttfamily\upshape jcarzon@umich.edu}}}

\author[J.~Gorham]{Jenna Gorham}
\address{Jenna Gorham\\  University of Arizona  \\ Tucson, Arizona \\
  \href{jgorham@email.arizona.edu}%
  {{\ttfamily\upshape jgorham@email.arizona.edu}}}

\author[K.~Jabbusch]{Kelly Jabbusch}
\address{Kelly Jabbusch\\ Department of Mathematics \& Statistics\\ University
  of Michigan--Dearborn \\ Dearborn, Michigan
 \\ 
  \href{mailto:jabbusch@umich.edu}%
  {{\ttfamily\upshape jabbusch@umich.edu}}}

\maketitle

\begin{abstract}


A toric code is an error-correcting code determined by a toric variety or its associated integral convex polytope. We investigate $4$- and $5$-dimensional toric $3$-fold codes, which are codes arising from polytopes in $\RR^3$ with four and five lattice points, respectively. By computing the minimum distances of each code, we fully classify the $4$-dimensional codes. We further present progress toward the same goal for dimension $5$ codes. In particular, we classify the $5$-dimensional toric $3$-fold codes arising from polytopes of width 1.
\end{abstract}

\tableofcontents

\section{Introduction}
\label{first section}

It is unknown in general whether an error correcting code will be ``good'' - say, in terms of the number of errors which it can dependably handle - until its parameters are computed, leading to an interest in classifying all the possible distinct codes. Some impressive results have come from particular toric codes, which Hansen introduced as the natural extension of Reed-Solomon codes in 1998 \cite{Hansen98}. By the correspondence between toric varieties and integral convex polytopes, one obtains a toric code $C_P$ determined by a polytope $P$. It is our imperative to be able to determine in general whether a code $C_P$ is the same as a code $C_{P'}$, where $P\neq P'$, in an essential way.

Many previous authors have written toward answering this question in particular cases. Little and Schwarz classified toric surface codes - toric codes corresponding to polytopes in $\RR^2$ - of dimensions up to 5 \cite{LSchwarz}. Luo, Yau, Zhang, and Zuo furthered this to toric surface codes up to 6 dimensions \cite{LYZZShort}, and later these same authors further classified most of the 7 dimension codes with Hussain \cite{HLYZZ}. The classification of 7 dimension codes was independently carried out by Cairncross, Ford, Garcia, and Jabbusch \cite{REU2019}. Departing slightly from the trend of these investigations, we present original theorems disambiguating the toric codes which are furnished by polytopes in $\RR^3$ with up to 4 lattice points. We refer to these as toric 3-fold codes of dimension 4 (as compared to \textit{toric surface codes}) to avoid confusion. We then make progress toward a classification of toric 3-fold codes of dimension 5 by considering those polytopes structurally very similar to those of dimension 4.

At a high level, a toric code $C$ is constructed by selecting a finite field $\FF_q$ and a convex lattice polytope $P$ in $[0,q-2]^m\subset\RR^m$. Codewords are obtained from evaluating linear combinations of monomials arising from the lattice points of $P$ at points in $(\FF_q^*)^m$, and so one can represent the code by a generator matrix whose row space is the dictionary of codewords. A code can be further described by parameters such as \textit{length}, the number of points in $(\FF_q^*)^m$, denoted $n$; \textit{dimension}, the number of lattice points $k=\lvert P\cap\ZZ^m\rvert$; and \textit{minimum distance}, the minimum Hamming distance between codewords, denoted $d$. Note that we refer to $k$ as the dimension of $C$ and denote it $dim(C)$, which is not to be confused with the dimension $m$ of the space in which $P$ is embedded. Two codes will be considered the same if they are \textit{monomially equivalent}, a term defined shortly.

Note that all toric codes obtained from polytopes in $\RR^3$ containing fewer than 4 lattice points are trivially classified because they can also be embedded in $\RR^2$ and have been previously studied. Likewise, for the purposes of this paper we will also overlook those codes furnished by polytopes of 4 or 5 lattice points in $\RR^3$ which can be embedded in $\RR^2$ since they are also well known.

More explicitly, the construction of a toric code is as follows. Selecting a finite field of order $q$ a prime power and a lattice convex polytope $P$ in $[0,q-2]^m$ with $k$ lattice points, one obtains a toric code $C_P$ by the vector space $Im(\epsilon)$, where the map $\epsilon: \mathcal{L}(P)\rightarrow (\FF_q)^n$ evaluates polynomials $f$ in 
\begin{equation*}
    \mathcal{L}(P) = \text{Span}_{\FF_q}\{\textbf{x}^p:p\in P\cap \ZZ^m\}
\end{equation*}
at each point in $(\FF_q)^m$. For example, let $q=11$ and let $P$ be the tetrahedron with vertices $(0,0,0)$, $(1,0,0)$, $(0,0,1)$, and $(1,1,1)$. (The significance of this choice is explained shortly.) Then we construct the collection of polynomials.
\begin{equation*}
    \mathcal{L}(P) = \{ax^0y^0z^0+bx^1y^0z^0+cx^0y^0z^1+dx^1y^1z^1:a,b,c,d\in\FF_{11}\}.
\end{equation*}
We may then select a polynomial $f$ in $\mathcal{L}(P)$ and evaluate it at each of the $n$ triples $(x,y,z)\in\left(\FF_{11}^*\right)^3$. Our resulting codeword is the $n$-tuple $(f(1,1,1),f(2,1,1),\ldots,f(11,11,11))$. This is more general than Reed-Solomon because the polynomials $f\in\mathcal{L}(P)$ may be in more than one indeterminate. The generator matrix $G$ of $C_P$ is the matrix $(\textbf{x}^p)$ with $\textbf{x}\in(\FF_q^*)^m$ and $p\in P\cap\ZZ^m$, where each column corresponds to a different point $\textbf{x}$ in the field. Then 
\begin{equation*}
    C_P = \{uG:u\in(\FF_q)^k\}.
\end{equation*}

Let $C_{P_1}$ and $C_{P_2}$ be toric codes with generator matrices $G_1$ and $G_2$ respectively.
\begin{definition} \label{def:moneq}
We say $C_{P_1}$ and $C_{P_2}$ are \textbf{monomially equivalent} if there exists an invertible $n \times n$ diagonal matrix $\Delta$ and an $n \times n$ permutation matrix $\Pi$ such that
\begin{equation} \label{eq:moneq}
    G_2 = G_1\Delta\Pi.
\end{equation}
\end{definition}
As a remark, if there exists an affine unimodular transformation $t:\ZZ^n\rightarrow\ZZ^n$ such that $t(P_1)=P_2$, then we say that $P_1$ and $P_2$ are \textit{lattice equivalent}, which implies that $C_{P_1}$ and $C_{P_2}$ are monomially equivalent \cite[Theorem 4]{LSchwarz}. For the toric 3-fold codes of dimension 4, we may make ready use of a classic theorem by White.
\begin{theorem}[\cite{white_1964}]
    Every 3-polytope with 4 integral lattice points of volume $t$ is lattice equivalent to an empty tetrahedron $$T(s,t) = \text{Conv}\{ (0,0,0), (1,0,0), (0,0,1), (s,t,1) \},$$ for some $s \in \ZZ$ where $\gcd(s,t)=1$. Moreover, $T(s, t)$ is lattice equivalent to $T(s', t)$ if and only if $s'=\pm s^{(\pm 1)}$ (mod $t$).
    
    \begin{figure}[t!]
        \centering
        \includegraphics[scale=0.4]{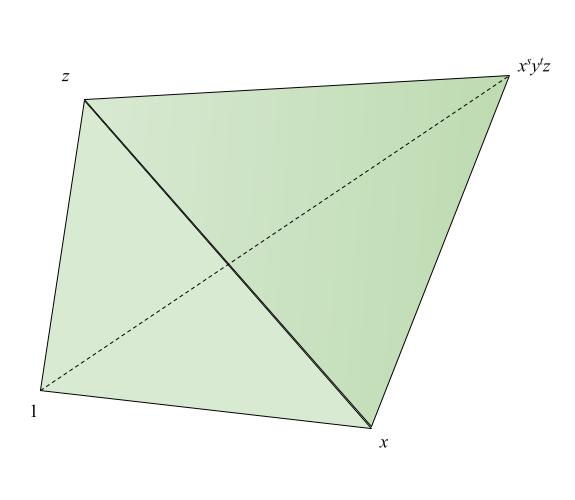}
        \caption{Monomials from the vertices of the empty tetrahedron $T(s,t)$.}
        \label{fig:my_label}
    \end{figure}
\end{theorem}
Thus if two 3-fold toric codes of dimension 4 are not monomially equivalent, then we know that they can be furnished by two empty tetrahedra with different values of the parameters $s$ or $t$.

\subsection{Statements of our theorems}

Therefore an outline of our approach to classifying four dimensional codes is as follows. Firstly, we provide some preliminary machinery in Section 3, most notably a formula for the minimum distances of four dimensional codes.

\begin{restatable}{thm}{mindist}
\label{thm:mindist4}
For any four dimensional toric code on an empty tetrahedron $T(s,t)$, the minimum distance is
\begin{equation*}
    d(C_{T(s,t)}) = 
    \begin{cases}
        (q-1)^3-(q-1)^2 &\mbox{if }\gcd(t,q-1)=1\\
        (q-1)^3-(q-1)(q-3)-q\gcd(t,q-1) &\mbox{else.}
    \end{cases}
\end{equation*}
\end{restatable}

Secondly, we inquire about the equivalence classes which we obtain when we fix $s$. To answer for this case, we prove the first classification theorem: 
\begin{restatable}{thm}{classone}
\label{thm:class1}
Let $C_1$ and $C_2$ be four dimensional toric codes on empty tetrahedra $T(s,t_1)$ and $T(s,t_2)$, respectively, over the field $\FF_q$. Then $C_1$ and $C_2$ are monomially equivalent if and only if $\gcd(t_1,q-1)=\gcd(t_2,q-1)$.
\end{restatable}
Lastly, if we conversely fix the parameter $t$, then we determine the equivalence classes in this case by proving the second classification theorem:
\begin{restatable}{thm}{classtwo}
\label{thm:class2}
Let $C_1$ and $C_2$ be four dimensional toric codes on empty tetrahedra $T(s_1,t)$ and $T(s_2,t)$, respectively, over the field $\FF_q$. If either of the following conditions hold true, then $C_1$ and $C_2$ are monomially equivalent:
\begin{itemize}
    \item $s_1 \equiv s_2 \bmod \gcd(t,q-1)$
    \item $s_1 \equiv \pm s_2^{\pm 1} \bmod t$
\end{itemize}
Furthermore, if neither of these hold, then the codes are not monomially equivalent.
\end{restatable}
The proofs of these theorems are undertaken in Section 4, and we thereby obtain a full classification of these toric codes.

Section 5 consists of describing the codes of dimension 5 and proving several results. In particular, we identify what monomial equivalence classes exist for codes arising from polytopes with 5 lattice points, all of which can be found in the union of two consecutive hyperplanes. (In other words, these are the polytopes of width 1 per the terminology used by Blanco and Santos \cite{BS4and5}.) However, we stop short of a full classification of these toric codes by not addressing those codes of width 2.

Throughout, much of our geometric notation and terminology will respect that of White and of Blanco and Santos. For example, the toric code furnished by the empty tetrahedron $T(s,t)$ will be written as $C_{T(s,t)}$, and in Section 5 we will use the invariant of a polytope with 5 lattice points which they call the \textit{signature} - an unordered integer pair - to refer to it and its furnished code. For example, if $P$ has signature $(i,j)$ and parameter values $s$ and $t$, then from it we obtain the code $C_{P^{(i,j)}(s,t)}$. We will then refer to the signature of the polytope furnishing a code as the signature of the code as well.

\section{Minimum distances for dimension 4}

A premium tool for evidencing the monomial inequivalence of two toric codes is the minimum distance invariant. If two codes are monomially equivalent, then they must have the same minimum distance. We now formally treat the dimension 4 case's minimum distances.
\begin{definition}
The \textit{minimum distance} of a code $C$ is 
\begin{equation*}
    d(C) = \underset{f\ne g\in\mathcal{L}(P)\setminus\{0\}}{min} h(\epsilon(f),\epsilon(g)),
\end{equation*}
where $h(\cdot,\cdot)$ is the Hamming metric.
\end{definition}
Recall that the Hamming metric $d$ counts the number of components which are distinct between two tuples of equal length. Alternatively, we will use the following fact:
\begin{remark}
An equivalent expression for the minimum distance of a toric code $C_P$ corresponding to a polytope $P \subseteq \RR^m$ is
\begin{equation} \label{eq:mindist}
    d(C) = (q-1)^m - \underset{0\neq f\in\mathcal{L}(P)}{max} Z(f),
\end{equation}
where $Z(f)$ is the number of zeros of $f$.
\end{remark}
Computing different maximum numbers of zeros for polynomials furnished by codes from different polytopes is a simple test that the codes are \textit{not} monomially equivalent. We also make use of a theorem by Soprunov and Soprunova.
\begin{theorem}[\cite{SS2}]
Let $P\subseteq\RR^m$ and $Q\subseteq\RR^n$. Then 
\begin{equation*}
    d(C_{P\times Q}) = d(C_P)d(C_Q).
\end{equation*}
\end{theorem}
For example, if $P$ is a polytope in $\RR^2$ and we take $Q$ to be $\{\textbf{0}\}$ in $\RR^1$, then the product $P\times Q$ is the embedding of $P$ in $\RR^3$. We can thereby determine the minimum distances of the degenerate cases of 3-fold toric codes. In particular, the four lattice equivalence classes of polygons with 4 lattice points represented in Figure \ref{fig:fourpoints} (each denoted $P^{(i)}_4$ for $i=1,2,3,4$ as in \cite{LSchwarz}) can be embedded in $\RR^3$ to obtain these pedantic exceptions. We will denote the $\RR^3$ embedding of $P^{(i)}_4$ as $P'^{(i)}_4$. Their resulting minimum distances are as listed:
\begin{itemize}
    \item $d(C_{P'^{(1)}_4}) = (q-1)^3-3(q-1)^2$
    \item $d(C_{P'^{(2)}_4}) = (q-1)^3-2(q-1)^2$
    \item $d(C_{P'^{(3)}_4}) = (q-1)^3-(2q-3)(q-1)$
    \item $d(C_{P'^{(4)}_4}) > (q-1)^3-(1+q+2\sqrt{q})(q-1)$
\end{itemize}
The fourth of these polygons is often called the ``exceptional triangle'' in this setting, possibly because its minimum distance formula remains elusive.

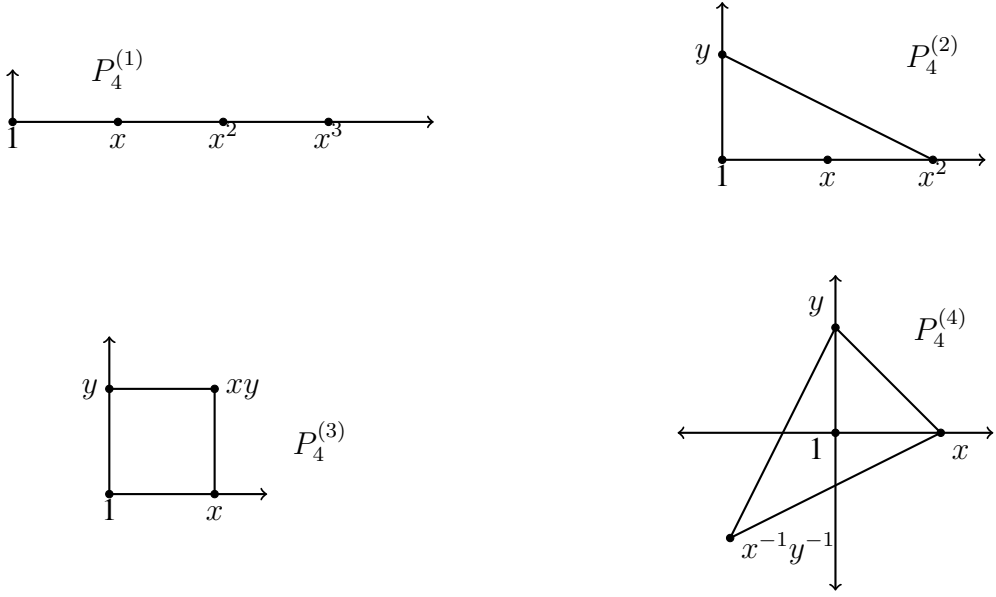
\begin{figure}[t]
    \label{fig:fourpoints}
    \centering
    \begin{minipage}{.5\textwidth}
        \centering
        \begin{tikzpicture}[scale=1.4]
            \draw [thick, <->] (0,.5) -- (0,0) -- (4,0);
            \node[above] at (0,-.35) {1};
            \node[above] at (1,-.35) {$x$};
            \node[above] at (2,-.35) {$x^2$};
            \node[above] at (3,-.35) {$x^3$};
            \foreach \x in {0,1,2,3}
                \draw[fill] (\x,0) circle [radius=0.035];
            \node at (1,.5) {$P^{(1)}_4$};
        \end{tikzpicture}
    \end{minipage}%
    \begin{minipage}{.5\textwidth}
        \centering
        \begin{tikzpicture}[scale=1.4]
            \draw [thick, <->] (0,1.5) -- (0,0) -- (2.5,0);
            \draw [thick] (0,1) -- (2,0);
            \node[above] at (0,-.35) {1};
            \node[above] at (1,-.35) {$x$};
            \node[above] at (2,-.35) {$x^2$};
            \node[left] at (0,1) {$y$};
            \foreach \x in {0,1,2}
                \draw[fill] (\x,0) circle [radius=0.035];
            \draw[fill] (0,1) circle [radius=0.035];
            \node at (2,1) {$P^{(2)}_4$};
    \end{tikzpicture}
    \end{minipage}
    
    \vspace{1cm}
    
    \begin{minipage}{.5\textwidth}
        \centering
        \begin{tikzpicture}[scale=1.4]
            \draw [thick, <->] (0,1.5) -- (0,0) -- (1.5,0);
            \draw [thick] (1,0) -- (1,1) -- (0,1);
            \node[above] at (0,-.35) {1};
            \node[above] at (1,-.35) {$x$};
            \node[left] at (0,1) {$y$};
            \node[right] at (1,1) {$xy$};
            \foreach \x in {0,1}
                \foreach \y in {0,1}
                    \draw[fill] (\x,\y) circle [radius=0.035];
            \node at (2,.5) {$P^{(3)}_4$};
        \end{tikzpicture}
    \end{minipage}%
    \begin{minipage}{.5\textwidth}
        \centering
        \begin{tikzpicture}[scale=1.4]
            \draw [thick, <->] (0,1.5) -- (0,0) -- (1.5,0);
            \draw [thick, <->] (0,-1.5) -- (0,0) -- (-1.5,0);
            \draw [thick] (0,1) -- (1,0) -- (-1,-1) -- (0,1);
            \node[above left] at (0,-.35) {1};
            \node[above right] at (1,-.35) {$x$};
            \node[above right] at (-1,-1.35) {$x^{-1}y^{-1}$};
            \node[above left] at (0,1) {$y$};
            \draw[fill] (0,0) circle [radius=0.035];
            \draw[fill] (0,1) circle [radius=0.035];
            \draw[fill] (1,0) circle [radius=0.035];
            \draw[fill] (-1,-1) circle [radius=0.035];
            \node at (1,1) {$P^{(4)}_4$};
    \end{tikzpicture}
    \end{minipage}
    
    \caption{Lattice equivalence class representatives of polygons with 4 lattice points.}
\end{figure}

We turn now to a formula for the minimum distance in dimension 4 of the nondegenerate cases.

\subsection{Formula for 4 dimensions}

In order to find the minimum distance of $C_{T(s,t)}$, we use (\ref{eq:mindist}). Then the polynomial with the maximum number of zeros will return us the minimum distance. According to the shape of $T(s,t)$, we know that such an $f\neq 0$ must be of the form 
\begin{equation} \label{eq:polyemptytetra}
    f(x,y,z) = a + bx + cz + dx^sy^tz,
\end{equation}
where $a$, $b$, $c$, and $d$ are not all $0$ in $\FF_q$. For simplicity, define 
\begin{equation*}
    f_x(y,z) = a + bx + cz + d_x xy^tz
\end{equation*}
when $x$ is fixed and $d_x=dx^{s-1}$. By doing so, we are able to count the vanishing points of $f$ for each $x\in\FF_q^*$ without concern for the value of $s$ since we know that each will also be a vanishing point for $f_x$ for some $x$ as well. Furthermore, each vanishing point of $f_x$ corresponds to one of $f$. It follows that 
\begin{equation*}
    Z(f) = \sum_{x} Z(f_x).
\end{equation*}

\mindist*

To prove the theorem, note that if $f(x,y,z)=1+x$ for example, then so long as $x=-1$ there are $q-1$ choices for each of $y$ and $z$ to obtain a vanishing point and so $\underset{f}{\max} \, Z(f)\ge(q-1)^2$ for any $s$ and $t$. We will first show that only when $t$ and $q-1$ are not relatively prime can we find another nonzero polynomial with more zeros than $1+x$.

\begin{prop}
If $f\in\mathcal{L}(T(s,t))$ has no zero coefficients, then
\begin{equation*}
    Z(f) \le q\gcd(t,q-1) + (q-1)(q-3).
\end{equation*}
Furthermore, there exists an $f$ for which this count is attained.
\end{prop}

\begin{proof}
Let $x=-ab^{-1}$. Then 
\begin{equation*}
    f_{-ab^{-1}}(y,z) = z(c+d_{-ab^{-1}}(-ab^{-1})y^t).
\end{equation*}
Write $N(x)$ for the number of values of $y$ for which $c+d_xxy^t=0$. In this case, we count $(q-1)N(-ab^{-1})$ zeros since $z$ is free to be any unit.

On the other hand, we also must consider points $(x,y,z)$ at which $f$ vanishes for which $x\ne-ab^{-1}$. Fix $x$. There are $(q-1)-N(x)$ choices of $y$ for which $c+d_xxy^t\ne0$. Therefore we have $\underset{x\ne -ab^{-1}}{\sum} [(q-1)-N(x)]$ combinations of $x$ and $y$ which allow us to set 
\begin{equation*}
    z = -(a+bx)(c+dx^sy^t)^{-1}.
\end{equation*}

Considering both of these subcases, then, the total number of zeros is
\begin{align} \label{eq:zf1}
    (q-1)N(-ab^{-1}) &+ \underset{x\ne -ab^{-1}}{\sum} [(q-1)-N(x)] \nonumber\\
        &= qN(-ab^{-1}) + (q-1)(q-2) - \underset{x}{\sum} N(x)\nonumber\\
        &= qN(-ab^{-1}) + (q-1)(q-3)
\end{align}
since certainly every value of $y$ will satisfy $c+d_xxy^t=0$ for some $x$, and so the summation term is equal to $q-1$. We obtain the most zeros of $f$ when we maximize $N(-ab^{-1})$ which occurs when $y^t=\frac{bc}{ad_x}$ has solutions. In particular, if $t$ and $q-1$ are not relatively prime, there are as many as $\gcd(t,q-1)$ solutions. Using this fact with (\ref{eq:zf1}), we conclude that if $f$ has no zero coefficients, then
\begin{equation*}
    Z(f)\le q\gcd(t,q-1) + (q-1)(q-3).
\end{equation*}
\end{proof}

So more generally we have
\begin{equation} \label{eq:maxzf1}
    \underset{0\neq f\in\mathcal{L}(P)}{\max} Z(f) \ge q\gcd(t,q-1) + (q-1)(q-3).
\end{equation}
As we will show, if $\gcd(t,q-1)>1$, then (\ref{eq:maxzf1}) gives the highest zero count; if else, then no polynomial vanishes at more than $(q-1)^2$ points. As with our example of $f=1+x$, the next proposition covers the other cases.

\begin{prop}
If $f\in\mathcal{L}(T(s,t))$ has at least one zero coefficient, then 
\begin{equation*}
    Z(f)\le (q-1)^2.
\end{equation*}
\end{prop}

\begin{proof}
We will show that any polynomial $f$ with at least one zero coefficient has at most $(q-1)^2$ vanishing points, using again the form of $f$ in (\ref{eq:polyemptytetra}) for consistency. The subcases this time are as follows:
\begin{itemize}
    \item Suppose $d=0$. Then $f(x,y,z)=a+bx+cz$. If $a+bx=0$, then $cz=0$ implies that $c=0$. Otherwise $a+bx\neq 0$, so we can let $z=-(a+bx)c^{-1}$. In either case we have $Z(f)=(q-1)^2$.
    \item Suppose $c=0$. (Note that because (\ref{eq:polyemptytetra}) has the aforementioned aesthetic symmetry between $x$ and $z$ after fixing $x$, this subcase is treated the same way as the subcase $b=0$.) Fixing $x=-ab^{-1}$, we have $f_{-ab^{-1}}(y,z)=d_{-ab^{-1}}(-ab^{-1})y^tz$. We obtain $d_{-ab^{-1}}(-ab^{-1})y^tz=0$, a contradiction, so for whatever choice of $y$ we can choose $z=-(a+bx)(d_xxy^t)^{-1}$ and have $Z(f)=(q-1)(q-2)$.
    \item Suppose $a=0$. Then $f(x,y,z)=bx+cz+dx^sy^tz$. Fix $y$. If $c+dx^sy^t=0$, then $bx=0$, a contradiction. Instead, so long as $x\neq -c(dy^t)^{-s}$, we can set $z=-bx(c+dx^sy^t)^{-1}$ and have $Z(f)=(q-1)(q-2)$.
\end{itemize}
By exhausting these three subcases we find that if $f$ has at least one zero coefficient, it has at most $(q-1)^2$ vanishing points; furthermore, we see that regardless of whether $t$ and $q-1$ are relatively prime we have
\begin{equation} \label{eq:maxzf2}
    \underset{0\neq f\in\mathcal{L}(P)}{\max} Z(f) \ge (q-1)^2.
\end{equation}
Thus we have our lower bound on $Z(f)$.
\end{proof}

What remains is a compilation of our results from the two cases in order to have our formula.

\mindist*

\begin{proof}[Proof of Theorem \ref{thm:mindist4}] The maximum number of vanishing points is equal to either the expression on the right of (\ref{eq:maxzf1}) or the one on the right of (\ref{eq:maxzf2}) since both expressions give upper bounds on the maximum number of vanishing points in the two distinct cases. It follows that if $t$ and $q-1$ are relatively prime and $q>2$, then (\ref{eq:maxzf2}) becomes an equality since 
\begin{equation*}
    q+(q-1)(q-3)<2(q-1)+(q-1)(q-3)=(q-1)^2.
\end{equation*}
If $\gcd(t,q-1)\ge 2$, then 
\begin{equation*}
    q\gcd(t,q-1)\ge 2q> 2(q-1),
\end{equation*}
so (\ref{eq:maxzf1}) becomes an equality. By use with (\ref{eq:mindist}), we conclude the proof of the proposition.
\end{proof}

\section{Classification of toric 3-fold codes of dimension 4}

In order to compare generator matrices, we introduce the following notation. Consider a partition of $(\FF_q^*)^3$ according to ordered pairs of the first and last coordinates and an additional condition involving the parameter $t$; in particular, define the partition 
\begin{equation*}
    \mathscr{P}_t = \{P(x_0,z_0,A) \ni (x,y,z)\in(\FF_q^*)^3:(x,z)=(x_0,z_0), y^t=A\}.
\end{equation*}
A cell in $\mathscr{P}_t$ may be empty, namely one for which $y^t=A$ has no solutions. The columns of $G_1$ and $G_2$ are indexed by the elements of $(\FF_q^*)^3$, so let $I_i(x_0,z_0,A)$ denote the set of column indices in $G_i$ which correspond to points in a cell in $\mathscr{P}_{t_i}$. For example, a representative column would be
\begin{equation*}
    G_1e_m = 
        \begin{bmatrix}
            1 \\
            x_0 \\
            z_0 \\
            x_0^sAz_0   
        \end{bmatrix},
\end{equation*}
so $m\in I_1(x_0,z_0,A)$.

\classone*

\begin{proof}[Proof of Theorem \ref{thm:class1}]
$(\Rightarrow)$ By contrapositive, supposing that $\gcd(t_1,q-1)\neq\gcd(t_2,q-1)$ immediately implies that $d(C_1)\neq d(C_2)$, so the codes are monomially inequivalent.

$(\Leftarrow)$ Let $l$ be the common $\gcd$, and let $\alpha$ be a primitive element of $\FF_q$. Denote by $G_1$ and $G_2$ generator matrices of $C_1$ and $C_2$, respectively, for which each column in either is led by a $1$ so that we need only worry to reorder the columns to determine equivalence. In other words, we know that equivalence means that there exist matrices - an invertible diagonal $\Delta$ and permutation $\Pi$ - for which (\ref{eq:moneq}) holds. However, at least the first rows are easily matched because they are all $1$s, so $\Delta=I$ and we need only confirm that there exists a satisfactory permutation $\Pi$.

We will show that $I_1$ and $I_2$ have the same cardinality for each $x_0,z_0,A\in\FF_q^*$. There exist integers $k_1$, $k_2$, and $k_3$ such that 
\begin{align*}
    t_1 &= k_1l\\
    t_2 &= k_2l\\
    q-1 &= k_3l
\end{align*}
Suppose $\alpha^m$ is a solution to $y^{t_1}=A$. Then $(\alpha^m)^{k_1l}=A$. Every solution can be found from $\alpha^m$ to be of the form $\alpha^{m+jk_3}$, $j=0,1,\ldots,l-1$. Similarly, each solution to $y^{t_2}=A$ is of the form $\alpha^{n+jk_3}$ for some $n$, so the solutions to either equation differ by multiplication by $\alpha^{n-m}$ and are thereby subject to a bijective correspondence. It follows that if $\lvert I_1\rvert>0$, then $\lvert I_1\rvert = \lvert I_2\rvert = l$. Supposing, on the other hand, that $y^{t_1}=A$ has no solutions in $\FF_q^*$ implies that $A$ is not in the range of the Frobenius automorphism $a\mapsto a^l$, so it is also the case that neither does $y^{t_2}=A$ have solutions. In any case, therefore, $\lvert I_1\rvert = \lvert I_2\rvert$.

Finally, we let $\Pi$ be any permutation on $(q-1)^3$ letters such that for all $x_0,z_0,A\in\FF_q^*$ we obtain 
\begin{equation*}
    \sum_{m\in I_1(x_0,z_0,A)} e_m = \sum_{n\in I_2(x_0,z_0,A)} \Pi e_n.
\end{equation*}
Then $\Pi$ is sufficient to satisfy our relation in (\ref{eq:moneq}) to confirm that $C_1$ and $C_2$ are monomially equivalent.
\end{proof}

Note that our approach to proving the Theorem \ref{thm:class1} is followed very closely again in Section 5 with Proposition \ref{prop:521}. Now we turn to proving our second classification theorem.

\classtwo*

\begin{proof}[Proof of Theorem \ref{thm:class2}]
$(\Rightarrow)$ Let $G_1$ and $G_2$ be generator matrices of $C_1$ and $C_2$, respectively. We will first show that either of the conditions implies monomial equivalence between $C_1$ and $C_2$.

Firstly, let $l=\gcd(t,q-1)$ and suppose $s_1\equiv s_2\bmod l$. Then there exist $j,k\in\ZZ$ such that $t=jl$ and $s=s_2-s_1=kl$. We seek to show that for any $x$ and $y_2$ there exists $y_1$ to satisfy 
\begin{equation} \label{eq:matchxy}
    x^{s_1}y_1^t = x^{s_2}y_2^t.
\end{equation}
For simplicity, let $n,i_2\in\ZZ$ give $x=\alpha^n$ and $y_2=\alpha^{i_2}$, where $\alpha$ is a primitive element of the field. Then by setting 
\begin{equation*}
    i_1 \equiv i_2 + nkj^{-1} \bmod \left(\frac{q-1}{l}\right),
\end{equation*}
we get that 
\begin{equation*}
    i_1(jl) \equiv i_2(jl) + n(kl) \bmod (q-1).
\end{equation*}
(Note that $j^{-1} \bmod \left(\frac{q-1}{l}\right)$ exists because $\gcd(j, \frac{q-1}{l})=1$.) This is equivalent to 
\begin{equation*}
    (\alpha^{i_1})^t = (\alpha^n)^s(\alpha^{i_2})^t
\end{equation*}
which is in turn equivalent to (\ref{eq:matchxy}), as desired. Using the same ordering as above, we have that $C_1$ and $C_2$ are monomially equivalent.

The second condition for monomial equivalence follows from White's theorem classifying four point polytopes in $\RR^3$. These two conditions are sufficient for monomial equivalence.

$(\Leftarrow)$ We now turn to prove that these conditions are necessary. Suppose $T(s_1,t)$ and $T(s_2,t)$ are not lattice equivalent but $C_1$ and $C_2$ are monomially equivalent. That is, the second condition does not hold. This supposition amounts to a stipulation that if $G_2=G_1\Delta\Pi$ as in Definition \ref{def:moneq}, then $\Delta=I$. (That is, a nontrivial diagonal matrix amounts to a reordering of the rows of $G$ (or lattice points in $P$). See \cite[Theorem 4]{LSchwarz}.) Furthermore, assume that $s_2-s_1\not\equiv 0 \bmod \gcd(t,q-1)$. That is, the first condition does not hold. We proceed toward a contradiction. We obtain $s_2-s_1=kl+r$ for some $k$ and $0<r<l$, where again $l=\gcd(t,q-1)$.

Because there exists a permutation between the generator matrices of $C_1$ and $C_2$, we may define a map $\FF_q^*\rightarrow\FF_q^*$ which sends $y_2$ to the value of $y_1$ satisfying
\begin{equation*}
    y_1^t=\alpha^{s_2-s_1}y_2^t
\end{equation*}
where $\alpha$ is a primitive element of $\FF_q^*$. Let $\phi$ denote the Frobenius automorphism $y\mapsto y^t$. Since $1\in Im(\phi)$ but $1\notin \alpha^r Im(\phi)$ for $r<l$, we obtain a contradiction. We conclude that the two conditions stated in Theorem \ref{thm:class2} are both sufficient and necessary for monomial equivalence.
\end{proof}

For completeness, we will exhibit that $C_1$ and $C_2$ are monomially equivalent if $s_1 \equiv \pm s_2^{\pm 1} \bmod t$ by explicit affine unimodular transformations.

\textsc{Case 1.} Suppose $s_1\equiv s_2 \bmod t$. Then $\frac{s_1-s_2}{t}\in\ZZ$. We can make use of the assignment rule 
\begin{equation*}
    \begin{bmatrix}
            x \\
            y \\
            z  
    \end{bmatrix}
    \mapsto
    \begin{bmatrix}
            1 & \frac{s_1-s_2}{t} & 0 \\
            0 & 1 & 0 \\
            0 & 0 & 1 
        \end{bmatrix}
    \begin{bmatrix}
            x \\
            y \\
            z  
    \end{bmatrix}
\end{equation*}
which maps each of the vertices of $T(s_2,t)$ to those of $T(s_1,t)$. Thus the polytopes are lattice equivalent.

\textsc{Case 2.} Suppose $s_1\equiv s_2^{-1} \bmod t$. Then $\frac{-s_1s_2+1}{t}\in\ZZ$. This time we will use the assignment rule 
\begin{equation*}
    \begin{bmatrix}
            x \\
            y \\
            z  
    \end{bmatrix}
        \mapsto
        \begin{bmatrix}
            s_1 & \frac{-s_1s_2+1}{t} & 0 \\
            t & -s_2 & 0 \\
            0 & 0 & -1 
        \end{bmatrix}
        \begin{bmatrix}
            x \\
            y \\
            z  
        \end{bmatrix}
            +
        \begin{bmatrix}
            0 \\
            0 \\
            1  
        \end{bmatrix}.
\end{equation*}
Once again, our map is an affine unimodular transformation and sends the vertices of $T(s_2,t)$ to those of $T(s_1,t)$, so again we have lattice equivalence.

\textsc{Case 3.} Suppose $s_1\equiv -s_2$. Then $\frac{s_1+s_2}{t}\in\ZZ$. Thus we can define a similar transformation mapping vertices of $T(s_2, t)$ to those of $T(s_1, t)$ as follows:
\begin{equation*}
    \begin{bmatrix}
            x \\
            y \\
            z
    \end{bmatrix}
        \mapsto
        \begin{bmatrix}
            -1 & \frac{s_1+s_2}{t} & -1\\
            0 & 1 & 0 \\
            0 & 0 & 1
        \end{bmatrix}
         \begin{bmatrix}
            x \\
            y \\
            z  
        \end{bmatrix}
            +
        \begin{bmatrix}
            1 \\
            0 \\
            0  
        \end{bmatrix}.
\end{equation*}
Thus, once again, the polytopes are lattice equivalent.

\textsc{Case 4.} Suppose $s_1 \equiv -s_2^{-1}$. Then a transformation can be constructed from lattice points of $T(s_2, t)$ to those of $T(s_1, t)$ by a composition of those mappings defined in cases 2 and 3.

The first of these affine unimodular transformations was also presented by Blanco and Santos \cite{BS4and5}. The following corollary is a consequence of the first condition listed in Theorem \ref{thm:class2}, namely that $s_1\equiv s_2\bmod\gcd(t,q-1)$.

\begin{cor}
If $\gcd(t,q-1)=1$, then $C_{T(s_1,t)}$ and $C_{T(s_2,t)}$ are monomially equivalent.
\end{cor}

\begin{proof}
Since $s_1\equiv s_2 \bmod 1$ trivially (regardless of $s_1$ and $s_2$), the codes are monomially equivalent by Theorem \ref{thm:class2}.
\end{proof}

\section{On dimension 5 codes}

There are many more polytopes up to lattice equivalence in $\RR^3$ with 5 lattice points than there are with 4. Two new invariants which we use to distinguish between those which are not lattice equivalent are the \textit{width} and \textit{signature}. The width of a polytope $P\subseteq \RR^3$ is the minimum number of consecutive hyperplanes in $\RR^3$ which contain $P$. The signature of $P$ is a statement about its affine dependence. There are four signatures of width 1 polytopes with 5 lattices, listed in Table \ref{tab:1width}. 

\begin{table}[b!]
    \centering
    \begin{tabular}{|c|c|c|}
        \hline
        Sig. & Volume & Representative \\\hline
        (2,1) & (-2$t$, $t$, 0, $t$, 0) & (0, 0, 0), (1, 0, 0), (0, 0, 1), (-1, 0, 0), ($s$, $t$, 1) \\
         & $0 \le s\le \frac{t}{2}$ &  \\\hline
        (2,2) & (-1, 1, 1, -1, 0) & (0, 0, 0), (1, 0, 0), (0, 1, 0), (1, 1, 0), (0, 0, 1) \\\hline
        (3,1) & (-3, 1, 1, 1, 0) & (0, 0, 0), (1, 0, 0), (0, 1, 0), (-1, -1, 0), (0, 0, 1) \\\hline
        (3,2) & ($-s-t$, $s$, $t$, 1, -1) & (0, 0, 0), (1, 0, 0), (0, 1, 0), (0, 0, 1), ($s$, $t$, 1) \\
         & $0<s\le t$ &  \\\hline
    \end{tabular}
    \caption{Five-point polytopes of width 1 \cite{BS4and5}.}
    \label{tab:1width}
\end{table}

Refer to the lattice points of $P$ by $p_i$ and the volumes of the convex hulls of $\{p_j\}_{j\ne i}$ by $w_i$, $i=1,\ldots,5$. The $w_i$ are possibly negative depending on orientation. Then there is an affine dependence 
\begin{equation} \label{eq:affdep}
    \sum_{k=1}^5 (-1)^{k-1}w_k,\hspace{2mm}\text{ and }\hspace{2mm}\sum_{k=1}^5 (-1)^{k-1}w_k p_k.
\end{equation}
The signature of $P$ is $(i,j)$ if the dependence (\ref{eq:affdep}) has $i$ positive coefficients and $j$ negative ones. (Thus the signature is unordered because the equations in (\ref{eq:affdep}) may be multiplied by $-1$.)

We summarize the lattice equivalence classes of polytopes in this case with Table \ref{tab:2width}. Regarding the rows for signatures $(2,1)$ and $(3,2)$, note that $s$ and $t$ must be relatively prime. Thus there are an infinite number of equivalence classes with these signatures. However, we consider only finitely many when observing the toric codes over the field $\FF_q$ since there come constraints on the size of $P$ with the choice of $q$.

\begin{table}
    \centering
    \begin{tabular}{|c|c|c|}
        \hline
        Sig. & Volume & Representative \\\hline
        (3,1) & (-9, 3, 3, 3, 0) & (0, 0, 0), (1, 0, 0), (0, 1, 0), (-1, -1, 0), (1, 2, 3) \\\hline
        (4,1) & (-4, 1, 1, 1, 1) & (0, 0, 0), (1, 0, 0), (0, 0, 1), (1, 1, 1), (-2, -1, -2) \\\hline
        (4,1) & (-5, 1, 1, 1, 2) & (0, 0, 0), (1, 0, 0), (0, 0, 1), (1, 2, 1), (-1, -1, -1) \\\hline
        (4,1) & (-7, 1, 1, 2, 3) & (0, 0, 0), (1, 0, 0), (0, 0, 1), (1, 3, 1), (-1, -2, -1) \\\hline
        (4,1) & (-11, 1, 3, 2, 5) & (0, 0, 0), (1, 0, 0), (0, 0, 1), (2, 5, 1), (-1, -2, -1) \\\hline
        (4,1) & (-13, 3, 4, 1, 5) & (0, 0, 0), (1, 0, 0), (0, 0, 1), (2, 5, 1), (-1, -1, -1) \\\hline
        (4,1) & (-17, 3, 5, 2, 7) & (0, 0, 0), (1, 0, 0), (0, 0, 1), (2, 7, 1), (-1, -2, -1) \\\hline
        (4,1) & (-19, 5, 4, 3, 7) & (0, 0, 0), (1, 0, 0), (0, 0, 1), (3, 7, 1), (-2, -3, -1) \\\hline
        (4,1) & (-20, 5, 5, 5, 5) & (0, 0, 0), (1, 0, 0), (0, 0, 1), (2, 5, 1), (-3, -5, -2) \\\hline
    \end{tabular}
    \caption{Five-point polytopes of width 2 \cite{BS4and5}.}
    \label{tab:2width}
\end{table}

\begin{figure}[t!]
    \centering
    \begin{minipage}{.5\textwidth}
      \centering
      \includegraphics[scale=0.35]{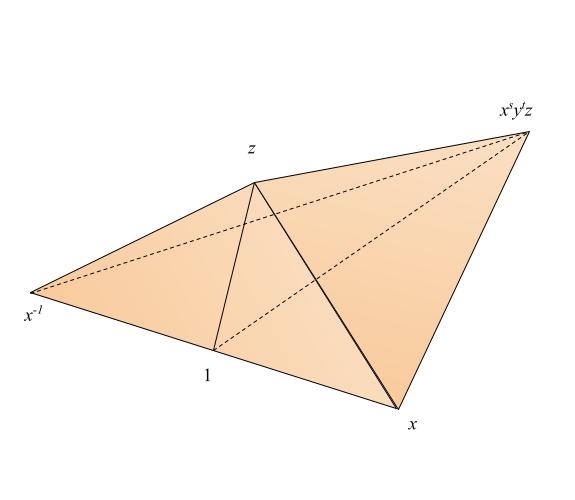}
    \end{minipage}%
    \begin{minipage}{.5\textwidth}
      \centering
      \includegraphics[scale=0.35]{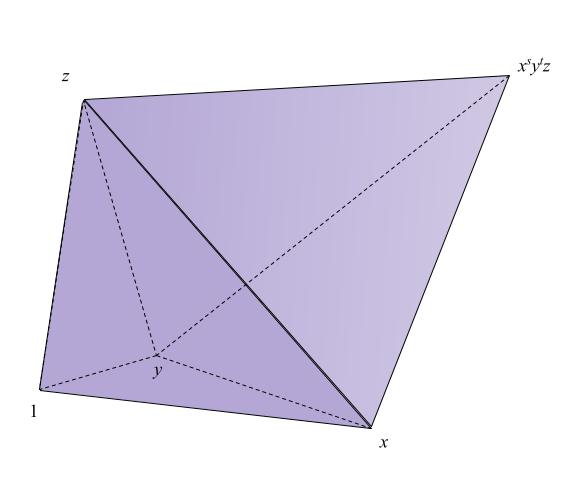}
    \end{minipage}
    \caption{Five point polytopes with signatures $(2,1)$ (left) and $(3,2)$ (right).}
\end{figure}

\subsection{Minimum distances for dimension 5}

We now offer and prove the following statements on the minimum distances of dimension 5 toric codes from only the polytopes of width 1.

\begin{theorem}[Minimum Distances for Five Dimensions] \label{thm:mindist5}
For the five point polytopes of width 1, we obtain the following minimum distance formulae and bounds:
    \begin{align}
        d(C_{P^{(2,1)}(s,t)}) &= (q-1)^3-2(q-1)^2 \label{eq:mindist21}\\
        d(C_{P^{(2,2)}}) &= (q-1)^3 - (2q^2 -5q + 3) \label{eq:mindist22}\\
        d(C_{P^{(3,1)}}) &\ge (q-1)^3 - (q-1)(1+q+2\sqrt{q}) \label{ineq:mindist31}\\
        d(C_{P^{(3,2)}(s,t)}) &\ge (q-1)^3-(q-2)^2 - (s+t)q \label{ineq:mindist32down}\\
        d(C_{P^{(3,2)}(s,t)}) &\le (q-1)^3 - (q-1)(q-3) - q\gcd(s+t, q-1) \label{ineq:mindist32up}
    \end{align}
\end{theorem}

\begin{proof}
The cases (\ref{eq:mindist22}) and (\ref{ineq:mindist31}) are the simplest cases. Recall that the minimum distance of the toric code from $P'^{(3)}_4$, the unit square in $\RR^3$, is $(q-1)^3-(2q-3)(q-1)$ and that from the exceptional triangle $P'^{(4)}_4$ is bounded by $d \geq (q-1)^3-(1+q+2\sqrt{q})(q-1)$. Since the five point polytopes with signatures $(2,2)$ and $(3,1)$ are exactly the unit pyramids over the unit square and the exceptional triangle, respectively, we have that their minimum distances are increased by a factor of $(q-1)$ \cite[Theorem 2.3]{SS2}.

Consider now (\ref{eq:mindist21}). Note that $(x+1)(x^{-1}-1)$ has $2(q-1)^2$ zeros. We claim that the polynomial 
\begin{equation*}
    f = A + Bx + Cx^{-1} + Dz + Ex^sy^tz
\end{equation*}
has no more than $2(q-1)^2$ zeros. If $E=0$, then any polynomial $f=A+Bx+Cx^{-1}+Dz $ will have no more than $2(q-1)^2$ zeros, since we have the unit pyramid over the line in the plane with three lattice points, which gives a code of minimum distance $(q-1)^3 - 2(q-1)^2$.   So we may assume that  $E \neq 0$.  
Call $P=A+Bx+Cx^{-1}$ and $Q=D+Ex^sy^t$.

We reduce this to two cases.

\textsc{Case 1.} Suppose $P$ has no solutions. Then $f=0$ only if $Q\ne 0$ so that we may select $x$ and $y$ freely and let $z=-PQ^{-1}$. Using the notation $N(x)$ similarly to its use in the proof of Theorem \ref{thm:class1}, giving the number of values of $y$ for which $Q_x=0$ for a fixed $x$, we see that there are 
\begin{align*}
    Z(f) &= \sum_{x\in\FF_q^*} [(q-1)-N(x)] \\
        &= (q-1)^2 - \sum_{x\in\FF_q^*} N(x) \\
        &= (q-1)^2-(q-1)\\
        &< 2(q-1)^2
\end{align*}
zeros of $f$, proving this case of the claim.

\textsc{Case 2.} Suppose $P$ has two solutions. Call them $x_1$ and $x_2$. Then $f=0$ if $x=x_1$, $y$ is one of $N(x_1)$ choices so that $Q=0$, and $z$ is anything at all or likewise if $x=x_2$ and $y$ is one of $N(x_2)$ choices for the same reason; or if $Q\ne 0$ as in Case 1. Therefore we have
\begin{align*}
    Z(f) &= (q-1)(N(x_1)+N(x_2)) + \sum_{x\ne x_1,x_2} [(q-1)-N(x)]\\
        &= q(N(x_1)+N(x_2)) + (q-1)(q-4)
\end{align*}
zeros. Recall that $N(x)=0$ or $d=\gcd(t,q-1)$. Note also that since $d<q-1$ and $d\mid (q-1)$, we know that $d\le\frac{q-1}{2}$ and so
\begin{align*}
    Z(f) &= q(N(x_1)+N(x_2)) + (q-1)(q-4)\\
        &\leq 2qd + (q-1)(q-4)\\
        &\leq q(q-1) + (q-1)(q-4)\\
        &=2(q-1)(q-2)\\
        &< 2(q-1)^2.
\end{align*}
This proves our claim for this case as well.

We turn to (\ref{ineq:mindist32down}).  We consider a polynomial of the form 
\begin{equation*}
    f = A + Bx + Cy + Dz + Ex^sy^tz
\end{equation*}
and let $P=A + Bx + Cy$ and $Q=D+Ex^sy^t$ be the same as above. Let $N$ denote the number of pairs $(x,y)$ for which $P$ and $Q$ are both $0$.

\textsc{Case 1.} Suppose $A\ne 0$. There are $(q-2)-N$ pairs $(x,y)$ which give $P=0\ne Q$ (since we cannot have $x=AB^{-1}$). There are $(q-1)-N$ pairs which give $P\ne 0=Q$. Therefore the number of zeros of $f$ are the sum of those giving $P\ne 0 \ne Q$, which are
\begin{equation*}
    (q-1)^2 - [(q-1)-N] - [(q-2)-N] - N = (q-2)^2+N,
\end{equation*}
and those giving $P=0=Q$, which are $N(q-1)$. Thus there are $(q-2)^2 + Nq$ zeros.

\textsc{Case 2.} Suppose $A=0$. This time, $(q-1)-N$ pairs $(x,y)$ give either $P=0\ne Q$ or $P\ne 0=Q$, and so there are $(q-1)(q-3)+Nq<(q-2)^2+Nq$ zeros.

Note that $N$ can be found as the number of solutions $x$ to the equation $(A'+x)^tx^s=B'$ (where $A'=AB^{-1}$ and $B'=(-1)^{t+1}(CB^{-1})^tDE^{-1}$) which has degree $s+t$. It follows that $d\leq (q-1)^3 - (q-2)^2 - (s+t)q$. 

Lastly, we turn to (\ref{ineq:mindist32up}). Using the same counts of zeros in cases 1 and 2 from (\ref{ineq:mindist32down}), we have
\begin{equation*}
    \underset{0\ne f\in\mathcal{L}(P)}{\max} Z(f) \geq (q-1)(q-3)+Nq.
\end{equation*}
Since $N\le\gcd(s+t,q-1)$ is an attainable bound, we obtain our upper bound on the minimum distance. Thus concludes the proof of all formulae and bounds on the minimum distance.
\end{proof}

Proceeding as we have before, we use Theorem \ref{thm:mindist5} to complete a classification of codes furnished by the polytopes with which the theorem is concerned. We begin by providing propositions akin to our main classification theorems to treat codes furnished by polytopes of signature $(2,1)$. Then we classify codes furnished by polytopes of signature $(3,2)$. Finally, we show that codes furnished by polytopes of width 1 and different signatures are not monomially equivalent.

\subsection{Classification for dimension 5}

\begin{prop} \label{prop:521}
Let $C_1$ and $C_2$ be toric codes on five point polytopes $P^{(2,1)}(s,t)$ and $P^{(2,1)}(s,t')$, respectively, over the field $\FF_q$. Then $C_1$ and $C_2$ are monomially equivalent if and only if $\gcd(t,q-1)=\gcd(t',q-1)$.
\end{prop}

\begin{proof}
($\Leftarrow$) As was the case with Theorem \ref{thm:class1}, we have monomial equivalence when we find a permutation matrix $\Pi$ which matches the columns corresponding to $I_1(x_0,z_0,A)$ with those of $I_2(x_0,z_0,A)$ for all $(x_0,z_0,A)\in(\FF_q^*)^3$. Indeed, since we are only including in the generator matrices the additional row determined by $\epsilon(x^{-1})$, we may repeat that previous argument exactly to conclude the sufficiency of this condition.

($\Rightarrow$) We don't have the luxury of a discerning minimum distance formula in this case, so instead suppose $A$ is in the image of both of the automorphisms $a\mapsto a^t$ and $a\mapsto a^{t'}$. We cannot create the desired permutation $\Pi$ matching the $\gcd(t,q-1)$ columns $I_1(x_0,z_0,A)$ to the $\gcd(t',q-1)$ columns $I_2(x_0,z_0,A)$, and so $C_1$ and $C_2$ are not monomially equivalent.
\end{proof}

\begin{prop}
Let $C_1$ and $C_2$ be toric codes on five point polytopes $P^{(2,1)}(s,t)$ and $P^{(2,1)}(s',t)$, respectively, over the field $\FF_q$. Then $C_1$ and $C_2$ are monomially equivalent if and only if $s\equiv s' \bmod\gcd(t,q-1)$.
\end{prop}

\begin{proof}
Unlike with Theorem \ref{thm:class2}, there are no nontrivial lattice equivalences between signature $(2,1)$ polytopes. Otherwise this proposition is proved identically to that of the first condition in the statement of that theorem.
\end{proof}

\begin{prop}
Let $C_1$ and $C_2$ be toric codes on five point polytopes $P^{(3,2)}(s,t)$ and $P^{(3,2)}(s',t')$, respectively, over the field $\FF_q$. Then $C_1$ and $C_2$ are monomially equivalent if and only if $(s,t)=(s',t')$.
\end{prop}

\begin{proof}
Suppose $C_1$ and $C_2$ are monomially equivalent so that their generators $G_1$ and $G_2$ are chosen such that $G_2=G_1\Pi$. Thus $G_2e_m=G_1\Pi e_m$ for each $m=1,\ldots,(q-1)^3$. In particular, we know that 
\begin{equation*}
    G_2e_m = \begin{bmatrix}
        1\\
        x\\
        y\\
        z\\
        x^sy^tz
    \end{bmatrix} = G_1\Pi e_m
\end{equation*}
for some $(x,y,z)\in\FF_q^*$. It follows that $x^sy^tz=x^{s'}y^{t'}z$; likewise for all $(x,y,z)$. Of course, by choosing $(x,1,1)$ we see that $x^s=x^{s'}$ for all $x$. Similarly, by choosing $(1,y,1)$ we see that $y^t=y^{t'}$ for all $y$. It follows that $s=s'$ and $t=t'$.
\end{proof}

\begin{prop}
Let $P_1$ and $P_2$ be five point polytopes of width 1. Then $C_{P_1}$ and $C_{P_2}$ are monomially equivalent only if $P_1$ and $P_2$ have the same signature.
\end{prop}

Note well that having the same signature is not sufficient for monomial equivalence.

\begin{proof}
This proposition amounts to stating that we obtain distinct codes when choosing polytopes of different signatures at the outset. We conclude this by first considering the minimum distance information that we have about all of these cases. Firstly, there is no risk of equivalence between polytopes of signatures $(2,1)$ and $(2,2)$ since for a fixed $q$ the furnished codes have different minimum distances given by the formulas.

If $q\ge 8$, then there can be no equivalence between $(2,1)$ and $(3,1)$ since $2(q-1)^2 > (q-1)(1+q+2\sqrt{q})$ in that case. To fully justify that codes from these polytopes will never be equivalent, a manual verification that minimum distances (or another invariant) differ for $q<8$ will suffice. Comparing $(2,1)$ and $(3,2)$, we require matching shape geometric parameters for both polytopes so that $x^sy^tz=x^{s'}y^{t'}z$ in general. However, since $x^{-1}\neq y\in\FF_q^*$ in general, we can have no equivalent codes. Comparing $(2,2)$ and $(3,1)$, if $q\ge 11$, then $(q-1)(2q-3) > (q-1)(q+q+2\sqrt{q})$, so we can have no equivalent codes. If $q<11$, a manual check shows that the codes are distinct. Comparing $(2,2)$ and $(3,2)$, since $x^sy^tz\neq xy$ in general, we can have no equivalent codes. Finally, comparing $(3,1)$ and $(3,2)$, since $x^sy^tz\neq x^{-1}y^{-1}$ in general, we can have no equivalent codes.
\end{proof}

In the cases where signature $(3,2)$ occurs in comparison with codes of another signature, note that our minimum distance intervals are not tight enough to conclude monomial inequivalence. However, it is easy to check that the generator matrices are distinct because $P^{(3,2)}(s,t)$, $P^{(3,1)}$, and $P^{(2,2)}$ all share four lattice points and $P^{(3,2)}(s,t)$ and $P^{(2,1)}(s,t)$ share four lattice points.

The patterns one will observe in the minimum distances of codes furnished by five point polytopes of width 2 are somewhat less tractable. It would be interesting to uncover what common and salient features there are of those polynomials' vanishing points in general. We end our classification of 3-fold toric codes at the injunction demarcated by this nontrivial question.

\section*{Acknowledgements}

This research was completed at the \textbf{REU Site: Mathematical Analysis and Applications at the University of Michigan-Dearborn}. We would like to thank the National Science Foundation (DMS-1950102); the National Security Agency (H98230-19); the College of Arts, Sciences, and Letters; and the Department of Mathematics and Statistics
for their support.

\bibliography{bibliography}

\end{document}